\documentclass[a4paper,12pt]{article}
\usepackage[T2A]{fontenc}
\usepackage[utf8]{inputenc}  
\usepackage[american, russian]{babel}
\usepackage[dvips]{graphicx}
\usepackage{color}
\usepackage{graphicx}
\usepackage{amsmath,amssymb,amsthm}
\usepackage{wrapfig}

\newtheorem{lemma}{Лемма}[section]

\newtheorem{thm}{Теорема}[section]
\newtheorem{cor}{Следствие}[section]
\newtheorem{prop}{Предложение}[section]
\newcounter{def}

\newcommand{\dist}{\mathop{\rm dist}\nolimits}
\newcommand{\cg}{\mathop{\rm CG}\nolimits}
\newcommand{\ko}{\mathop{\rm K}\nolimits}

\begin{document}

\title{Экстремальные задачи упаковок кругов на сфере и  неприводимые контактные графы}
\author{О. Р. Мусин\thanks{Работа выполнена при частичной поддержке гранта РФФИ  13-01-12458 и гранта NSF  DMS-1400876},\,  А. С. Тарасов}
\date{}
\maketitle
\begin{abstract}
Недавно, с точностью до изометрии, нами были перечислены все локально-жесткие упаковки  конгруэнтных кругов (сферических шапочек) на единичной сфере  с числом кругов $N <12$. Эта задача эквивалентна перечислению сферических неприводимых контактных графов. В этой статье мы показываем, что с помощью списка неприводимых контактных графов можно решать различные задачи об экстремальных упаковках таких как задача Таммеса для сферы и проективной плоскости, задача о наибольшем числе контактов у сферических упаковок, задачи Данцера и другие задачи о неприводимых контактных графах. 
\end{abstract}

\section{Введение}
Упаковка шаров в пространстве называется {\it жесткой} или, иногда, {\it локально-жесткой}, если они расположены так, что каждый шар из упаковки зажат соседями и его нельзя сдвинуть в сторону с тем, чтобы увеличить минимальное расстояние между центром этого шара и центрами других шаров.

Рассмотрим $N$ не пересекающихся шаров одинакового радиуса $r$ в ${\Bbb R}^3$, которые расположены так, что все они касаются одного (центрального) шара единичного радиуса. Обозначим через $P:=\{A_{1},\ldots,A_{N}\}$ точки касания внешними шарами центрального шара. Соединим точки $A_{i}$ и $A_{j}$ ребром (минимальной дугой большого круга), если соответствующие внешние шары касаются. Полученный граф будем называть {\it контактным} и обозначать $\cg(P)$.   Если же эта упаковка на ${\Bbb S}^2$ является локально-жесткой, то будем называть граф $\cg(P)$ {\it неприводимым}. 
Таким образом, задача изучения жестких упаковок локально сводится к изучению неприводимых графов.



Имеется связь этой геометрической задачи с другими задачами упаковки шаров. Одно из основных приложений вне математики - это материаловедение, где рассматриваются локально-жесткие упаковки твердыми телами и наночастицами (см. например, \cite{AMB,Appl04, Appl10}). Заметим также, что большинство конфигураций физических частиц, задающих минимум потенциальной энергии тоже являются локально-жесткими.  

В математике - В. Хабихт, К. Шютте, Б. Л. ван дер Варден, и Л. Данцер применяли неприводимые контактные графы к проблеме контактных чисел и проблеме Таммеса \cite{HabvdW, SvdW1, vdW, SvdW2, Dan}. Недавно с помощью этого метода мы решили проблему Таммеса для $N=13$ и $N=14$ \cite{MT,MT14}. 

В этой статье мы показываем, что с помощью списка неприводимых контактных графов можно решать различные задачи об экстремальных упаковках.  Здесь мы рассматриваем задачи Таммеса для сферы и проективной плоскости, задачу о наибольшем числе ребер у сферического контактного графа, задачи Данцера и другие.

\section{Неприводимые контактные графы}

\subsection{Основные определения}
Обозначим через ${\Bbb S}^{2}$ единичную сферу в ${\Bbb R}^3$. Для точек $x$ и $y$ на сфере, $\dist(x,y)$ - это расстояние в угловом измерении.  

Пусть $X$ - конечное подмножество единичной сфере ${\Bbb S}^{2}$. Обозначим 
$$\psi(X):=\min\limits_{x,y\in X}{\{\dist(x,y)\}}, \mbox{ где } x\ne y.$$

Пусть $d_N$ обозначает  наибольшее значение $\psi(X)$, которое может достигаться для 
$X\subset{\Bbb S}^{2}$  с $|X|=N$, т. е. 
$$
d_N:=\max\limits_{X\subset{\Bbb S}^2}{\{\psi(X)\}}, \, \mbox{ при } \;  |X|=N.
$$

\noindent{\bf Контактные графы.} Как и выше, $X$ - конечное подмножество ${\Bbb
S}^2$.  {\it Контактным графом} $\cg(X)$ называется граф на ${\Bbb S}^2$ с вершинами в $X$ и ребрами (дугами) $xy, \, x,y\in X$, минимальной длины, т. е. с $\dist(x,y)=\psi(X)$.

\medskip

\noindent{\bf Сдвиг вершины.}  
Будем говорить, что вершину $x\in X$ можно {\it сдвинуть в сторону}, если в любой открытой окрестности $x$ найдется такая точка  $x'\in {\Bbb S}^2$, что $$\dist(x',X\setminus\{x\})>\dist(x,X\setminus\{x\}).$$

\medskip

\noindent{\bf Неприводимые графы.}
Назовем контактный граф $\cg(X)$ {\it неприводимым}, если ни одну из его вершин нельзя сдвинуть в сторону.
Этот термин  используется в работах \cite{SvdW1,SvdW2,FeT,Dan}. 

Давайте обозначим через $J_N$ семейство всех конечных множеств $X$ в ${\Bbb S}^2$ с $|X|=N$ таких что $\cg(X)$ является неприводимым.  

\medskip

\noindent{\bf  Д-неприводимые графы.}
Л. Данцер \cite[Sec. 1]{Dan} определил следующую операцию. Пусть 
$x,y,z$ - вершины $\cg(X)$ с $\dist(x,y)=\dist(x,z)=\psi(X)$. Обозначим через  $x^0$ зеркальный образ точки $x$ 
относительно дуги большого круга $yz$ (см. Рис.~\ref{fig3}). Мы назовем эту операцию 
 {\it Д-отражением}, если $\dist(x^0,X\setminus\{x,y,z\}) > \psi(X)$.

Неприводимый контактный граф $\cg(X)$ будем называть Д-неприводимым, если он не допускает ни одного Д-отражения. 

\medskip

\begin{figure}[h]
\begin{center}
\includegraphics[clip,scale=1]{pics/13-6.mps}
\end{center}
\caption{Д-отражение}
\label{fig3}
\end{figure}

\medskip

\noindent{\bf Максимальные графы.} Предположим, что для $X\subset{\Bbb S}^{2}$  с $|X|=N$ имеет место равенство $\psi(X)=d_N$. Будем тогда называть контактный граф $\cg(X)$ - {\it максимальным.}

\subsection{Свойства неприводимых контактных графов.}

Здесь мы рассмотрим такие подмножества $X\subset {\Bbb S}^2$, что граф $\cg(X)$ является неприводимым, т.е. $X\in J_N$.   
Следующие свойства неприводимых графов были опубликованы в работах \cite{SvdW1}, \cite{Dan}, и \cite{BS,BS14} (см. также  \cite[Глава VI]{FeT}). 

 Пусть $a,b,x,y\in X$ с $\dist(a,b)=\dist(x,y)=\psi(X)$.
Тогда кратчайшие дуги ${ab}$ и ${xy}$ не пересекаются. В противном случае,
длина одной из дуг $ax, ay, bx, by$ будет меньше чем $\psi(X)$. Из этого вытекает 
\begin{prop} Если $X$ - конечное подмножество ${\Bbb S}^2$, то  $\cg(X)$ является планарным графом.
\end{prop}

\begin{prop} Если  $X\in J_N$, то все грани $\cg(X)$ являются выпуклыми в ${\Bbb S}^2$. 
\end{prop}



\begin{prop} Если $X$ является максимальным, то для  $N>5$ граф $\cg(X)$ является Д-неприводимым и, в частности, неприводимым.
\end{prop}

\begin{prop} Если $X\in J_N$,  то степени вершин графа $\cg(X)$ могут быть только  $0$ (изолированные вершины), $3$, $4$, или  $5$.
\end{prop}

\begin{prop} Если $X\in J_N$, то грани $\cg(X)$ являются многоугольниками не более чем с   $\lfloor2\pi/\psi(X)\rfloor$ вершинами.
\end{prop}

Следующее свойство было найдено Бёрёцким и Сабо	 {\cite[Lemma 8 and Lemma 9(iii)]{BS}}. 

\begin{prop}  Пусть $X\in J_N, \, N>10$. Если $\cg(X)$ содержит изолированные вершины, то эти вершины лежат внутри граней   $\cg(X)$ с шестью или более вершинами. Более того, внутри шестиугольника не может лежать две изолированные вершины.  
\end{prop}

Комбинируя эти предложения вместе, получаем следующие комбинаторные свойства неприводимых контактных графов: 
\begin{cor}\label{cor1} Если $X\in J_N$, то $G:=\cg(X)$ удовлетворяет следующим свойствам  
\begin{enumerate}
\item $G$ является планарным графом;
\item У любой вершины $G$ степень равна $0,3,4,$ или $5$;
\item Если $N>10$ и у  $G$ есть изолированная вершина $v$, то $v$ лежит в грани с $m\ge 6$ вершинами. Гексагональная грань     $G$ не может содержать две изолированные вершины.
\end{enumerate}
\end{cor}

\subsection{Работа Л. Данцера по неприводимым графам} 

В работе \cite{Dan} Людвиг Данцер приводит решение проблемы Таммеса для $N=10$ и $N=11$. Эта статья - английский перевод его  докторской диссертации: ``Endliche Punktmengen auf der 2-sph\"are mit m\"oglichst gro{\ss}em Minimalabstand'', Universit\"at G\"othingen, 1963. В этой работе, в частности,  к понятию приводимости было добавлено понятие Д-отражения и Д-приводимости. (В статье применяется другая терминология и мы используем здесь Д в честь Л. Данцера.) 

\medskip


В своей работе \cite{Dan} Данцер также приводит список Д-неприводимых графов для   $6\le N \le 10$. Так как максимальные контактные графы являются Д-неприводимыми, то в этот список включены и графы, дающие решение проблемы Таммеса для этих $n$.  Заметим, что для случая $N=11$ Данцер рассматривал только максимальные контактные графы.

\subsection{Перечисление неприводимых контактных графов} Пусть конечное множество точек $X\subset {\Bbb S}^2$ такое, что контактный граф $\cg(X)$ является неприводимым. В Следствии 2.1 мы собрали вместе комбинаторные свойства $\cg(X)$.  Имеется целый ряд геометрических свойств этих графов. 

Напомним, что все грани $\cg(X)$ являются выпуклыми (Предложение 2.2). Поскольку все ребра $\cg(X)$ одинаковой длины  $\psi(X)$, то его грани - выпуклые равносторонние сферические многоугольники с числом вершин не превосходящим $\lfloor2\pi/\psi(X)\rfloor$.  

Рассмотрим теперь планарный граф $G$, с заданными гранями $\{F_k\}$, который удовлетворяет Следствию 2.1. Мы будем рассматривать вложения этого графа в ${\Bbb S}^2$ как неприводимого контактного графы  $\cg(X)$ для некоторого $X\subset {\Bbb S}^2$.

Вложение графа $G$ в ${\Bbb S}^2$ однозначно задается следующим набором параметров (переменных):\\
(i) Длина ребра  $d$;\\
(ii) Наборы углов $u_{ki}$, $i=1,\ldots,m_k$ граней $F_k$. (Здесь $m_k$ обозначает число вершин у грани $F_k$.)

В нашей работе по перечислению графов были рассмотрены основные геометрические соотношения между этими параметрами (\cite[Предложение 4.1]{MT2013}).

\medskip

Алгоритм перечисления неприводимых контактных графов состоит из двух частей:\\
(I) На первом этапе составляется список $L_{N}$, состоящий из всех графов с $N$ вершинами и удовлетворяющий Следствию 2.1;

Чтобы создать список $L_{N}$ мы используем программу {\it plantri}
(см. \cite{PLA1,PLA2}). Эта программа является генератором не изоморфных планарных графов различных классов, включая триангуляции и другие разбиения на выпуклые многоугольники. 

Заметим. что число графов в $L_{N}$ с ростом $N$ быстро возрастает. Например, при $N=6,7,8$, $|L_N|=7,34, 257$, а уже 
$|L_{13}|=94754965$

\medskip

\noindent(II) Используя линейную аппроксимацию соотношений из \cite[Предложения 4.1]{MT2013}, из $L_{N}$ удаляются все графы, которые не могут быть вложены в сферу. Оставшиеся графы после дополнительной проверки с помощью солверов и оценки границ изменения параметров заносятся в список неприводимых контактных графов. 

\medskip

Основной результат работы \cite{MT2013} приведен в Приложении.

\section{Контактные числа и проблема Таммеса}

\subsection{Контактные числа} 
{\it Контактным числом} $k(n)$ называют наибольшее число не пересекающихся шаров одинакового радиуса в  ${\Bbb R}^n$, которые можно расположить так, чтобы все они касались одного (центрального) шара такого же радиуса. 

Очевидно, что $k(2)=6$. В трехмерном пространстве, в задаче о контактных числах спрашивается: ``Как много белых бильярдных шаров могут одновременно касаться черного бильярдного шара?''
   
Наиболее симметричная конфигурация, 12 бильярдных шаров вокруг одного, это когда центры 12 шаров расположены в вершинах правильного икосаэдра, а центральный шар расположен в центре икосаэдра. Однако, эти 12 внешних шаров не касаются друг друга и могут свободно перемещаться по поверхности центрального шара. Таким образом, возможно, что эти 12 шаров можно сдвинуть в одну сторону, так что найдется место для 13-го шара? 

Этот вопрос был предметом спора между И. Ньютоном и Д. Грегори в 1694 году. Ньютон считал, что $k(3)=12$, в то время как Грегори думал, что ответ может быть равен 13. Эту задачу Ньютона - Грегори часто называют {\it проблемой тринадцати шаров.}

Несложно видеть, что проблема тринадцати шаров сводится к следующей задаче: {\it Доказать, что  на единичной сфере ${\Bbb S}^2$ нельзя расположить 13 точек так, чтобы расстояния между ними были не меньше чем  $60^\circ$ в угловом измерении.} 

Проблема тринадцати шаров оказалось достаточно трудной и была решена только в 1953 году. К. Шютте и Б.Л. Ван дер Варден \cite{SvdW2} доказали, что Ньютон был прав и $k(3)=12$. Доказательство Шютте -- ван дер Вардена основано на неприводимых контактных графах. Ими было показано, что что на единичной сфере ${\Bbb S}^2$ не найдется контактного графа ${\Gamma}$ с ребрами одинаковой длины, которая не меньше чем  $60^\circ$.

В 1956 году Дж. Лич \cite{Lee} напечатал двухстраничный набросок элегантного доказательства. Это доказательство было приведено в первом издании известной книги М. Айгнера и Г. Циглера ``Доказательства из Книги'' \cite{AZ}. Однако эта глава была исключена из книги при втором издании, так как авторам не удалось привести подробное доказательство без громоздких вычислений, основанных на сферической тригонометрии. В последние 12 лет было опубликовано несколько новых решений этой старой проблемы 
\cite{Hs, Ma, Manew, Bor,Ans, Mus13}. 

Заметим, что проблема контактных чисел решена только для размерностей $n=3,4,8$  и $24$ (см. \cite{BDM, Mus1, Mus3, Mus4}). Доказательства в этих работах основаны на методе Дельсарта и его обобщениях.

\subsection{Проблема Таммеса} 
У проблемы 13 шаров имеется естественное обобщение: найти расположение множества $X$, состоящего из $N$ точек на  ${\Bbb S}^{2}$,  такое что минимальное расстояние между точками $X$ - максимально возможное. Эту задачу впервые поставил голландский ботаник Таммес  \cite{Tam} (см. также \cite[Section 1.6: Problem 6]{BMP}).

Задача Таммеса решена только для нескольких значений $N$: для $N=3,4,6,12$ ее решил Л. Фейеш Тот \cite{FeT0}; для 
$N=5,7,8,9$  - Шютте и ван дер Варден \cite{SvdW1}; для $N=10,11$ - Данцер \cite{Dan} (для  $N=11$ см. также \cite{Bor11}) и для $N=24$ - Робинсон \cite{Rob}. Недавно мы решили эту задачу для $N=13$ \cite{MT} и для $N=14$ \cite{MT14}.  

В работе Л. Фейеша Тота \cite{FeT0} (см. также его книгу \cite{FeT}) была найдена верхняя оценка для  $d_N$.   (Напомним, что $d_N$ обозначает  наибольшее значение $\psi(X)$, которое может достигаться для 
$X\subset{\Bbb S}^{2}$  с $|X|=N$.)
$$
d_N\leqslant \arccos{\frac{\ctg^2{w_N}-1}{2}}, \; \mbox{ где } \; w_N:=\frac{\pi N}{6N-12}. 
$$

Эта формула является точной для $N=3,4,6,12$ и решением проблемы Таммеса для этих $N$ соответственно являются правильный треугольник на экваторе, правильный тетраэдр, правильный октаэдр и правильный икосаэдр. 

Р. М. Робинсон \cite{Rob} обобщил оценку Фейеша Тота  и решил проблему Таммеса для $N=24$. 

Оказывается, что $d_5=d_6=90^\circ$ \cite{SvdW1, FeT}. При $N=5$ максимальное расположение на сфере не является единственным с точностью до изометрии. Если взять две точки расположенные в северном и южном полюсах сферы, а оставшиеся три точки расположить на экваторе так, чтобы (угловое) расстояния между ними было бы не меньше $90^\circ$, то получим максимальное расположение. 

Для решения проблемы Таммеса при $7\leqslant N \leqslant 11$ и при $N=13, 14$  применялись неприводимые контактные графы. Заметим, что контактный граф у максимальной конфигурации должен быть неприводимым.  Поэтому, максимальный граф находится среди неприводимых контактных графов. Более того, он является Д-неприводимым \cite{Dan}. 

Можно значительно сократить перебор и список ``допустимых'' неприводимых графов, если добавить условие: $d\geqslant\delta_N$, где $\delta_N$ обозначает длину ребра у гипотетического максимального графа. (Собственно, проблема состоит в доказательстве равенства $d_N=\delta_N$.) Большой набор примеров имеется в книге \cite[Глава VI, §4]{FeT} и в таблице Н. Слоэна:   http://neilsloane.com/packings/dim3/. В этой таблице приведены возможные максимальные конфигурации вплоть до $N=130$. 

Из Таблиц 7.1--7.6 в Приложении получаем следующую таблицу: 

\medskip

\begin{tabular}{ccccccc}
$N$ & $6$ & $7$ & $8$ & $9$ & $10$ & $11$\\
$I_N$ & $2$ & $2$ & $4$ & $10$ & $30$ & $38$\\
\end{tabular}

\medskip

Здесь $I_N$ обозначает число неприводимых контактных графов с $N$ вершинами. 

Поскольку, для $N=7,8,9$ числа $I_N$ не очень большие, то в работе \cite{SvdW1} удалось отбросить неподходящие графы без особой суеты. При $N=10$ число $I_N=30$, что довольно много. В этом случае, также как и при $N=11$, Л. Данцер рассматривал только Д-неприводимые контактные графы. В наших работах по проблеме Таммеса для $N=13$ \cite{MT} и для $N=14$ \cite{MT14} мы проводили компьютерный перебор подходящих, т. е. из списка $L_N$, Д-неприводимых контактных графов. 

Интересно, что $d_{11}=d_{12}=\arccos{(1/\sqrt{5}})$. Для  $N=11$ максимальная конфигурация получается из вершин правильного икосаэдра удалением одной из них. 

Заметим, что во всех решенных случаях проблемы Таммеса, максимальное расположение оказалось единственным с точностью до изометрии. Можно, конечно, предположить, что так будет для всех $N$. Однако, в задаче, аналогичной проблеме Таммеса, для плоского тора (периодические упаковки конгруэнтных кругов) при $N=7$ оказалось, что имеется три разные максимальные упаковки \cite{MN}. Это обстоятельство делает проблему Таммеса для сферы еще более интригующей. 

В заключении этого раздела отметим, что прямой подход по решению проблемы Таммеса, основанный на компьютерном перечислении неприводимых контактных графов, т. е. такой как в наших работах \cite{MT,MT14}, себя практически исчерпал. Для $N=15,16,...$ компьютерный перебор может занимать многие месяцы. Здесь нужны новые подходы и идеи.

\section{Проблема Таммеса для антиподальных конфигураций}

Если рассмотреть сферу ${\mathbb S}^d$ как множество единичных векторов $x$ в ${\mathbb R}^{d+1}$, то точки  $x$ и $y=-x$ называются {\it антиподальными}, а отображение 
 $x \to -x$ называется {\it антиподальным} на ${\mathbb S}^d$. Множество $X\subset {\mathbb S}^d$ называется  {\it антиподальным}, если антиподальное отображение переводит $X$ в себе, $X=-X$, иными словами, если $x\in X$, то и $-x\in X$.

Рассмотрим теперь проблему Таммеса для антиподальных множеств на ${\mathbb S}^2$. Пусть $X$ -- антиподальное множество на сфере.  Тогда $X$ содержит четное число точек и, стало быть, $|X|=2M, \, M=1,2,...$. Определим для таких множеств аналог величины $d_N$. 

$$
a_M:=\max\limits_{X=-X\subset{\Bbb S}^2}{\{\psi(X)\}}, \, \mbox{ при } \;  |X|=2M.
$$

{\it Проблема Таммеса для антиподальных множеств} состоит в том, чтобы найти все конфигурации  $X=\{x_1,-x_1,\ldots,x_M,-x_M\}$ на ${\mathbb S}^2$, чтобы $\psi(X)=a_M$. 

Если отождествить на сфере ${\mathbb S}^2$ антиподальные точки, то получим проективную плоскость ${\Bbb R\Bbb P}^2$. Поэтому, по сути, эта проблема о конфигурациях на проективной плоскости у которых минимальное расстояние между точками является максимально возможным. 

Рассмотрим простейшие свойства $a_M$ и оптимальных конфигураций. 

\begin{lemma}
	\begin{enumerate}
	\item $a_M\leqslant d_{2M}$;
	\item Если $|X|=2M, \, \psi(X)=d_{2M}$ и $X$ -- антиподальное множество, то $a_M = d_{2M}$.  Следовательно, если $X=-X$ и $X$ является решением проблемы Таммеса для $N=2M$, то  $X$ также является решением и задачи Таммеса для антиподальных конфигураций;
	\item Если $X\subset{\Bbb S}^2$ -- антиподальное множество, $|X|=2M$ и $\psi(X)=a_{M}$, то $\cg(X)$ является Д-неприводимым контактным графом. 
	\end{enumerate}
\end{lemma}
\begin{proof} Доказательства 1 и 2 вытекают непосредственно из определений.  
	
	3. Будем доказывать это утверждение от противного.  Для антиподального множества $X$ сдвиг вершины и Д-отражение мы будем делать одновременно для вершины $v$ и ее антипода $-v$, т. е. мы сохраняем антиподальность. Тогда, если конфигурация позволяет такую операцию, то сдвигая еще точки мы получим  новое антиподальное $X'$ у которого $\psi(X')>\psi(X)$, -- противоречие. 
\end{proof}

\begin{thm} Пусть $X_M\subset{\Bbb S}^2$ является решением задачи Таммеса для антиподальных конфигураций, т. е.  $\psi(X_M)=a_M$. Тогда 
	\begin{enumerate}
	\item $X_2$ -- множество вершин квадрата на экваторе, $a_2=90^\circ$;
	\item   $X_3$ -- множество вершин правильного октаэдра, $a_3=90^\circ$;
	\item   $X_4$ -- множества вершин куба, $a_4=\arccos{(1/3)}$;
	\item    $X_5$ -- множество, состоящее из пяти пар антиподальных вершин правильного икосаэдра, $a_5=\arccos{(1/\sqrt{5})}$.
	\item    $X_6$ -- множество вершин правильного икосаэдра, $a_6=\arccos{(1/\sqrt{5}})$.
\end{enumerate}
\end{thm}

\begin{proof} 1. В случае $M=2$ у нас имеется две пары антиподальных точек и поэтому эти четыре точки лежат на большой окружности. Следовательно, вершины квадрата являются оптимальным расположением этих точек.
	
Пункты 2 и 5 теоремы вытекают из Леммы 4.1 (2). 

Лемма 4.1 (3) позволяет выбрать $X_M$ из списка неприводимых контактных графов. Таблицы 7.3 и 7.5 доказывают пункты 3 и 4 теоремы. 
\end{proof}


\section{Задача о наибольшем числе контактов}

Пусть $X$ конечное  множество в произвольном метрическом пространстве ${\bf M}$ с расстоянием $\dist$. 
Как и раньше, 
$$\psi(X):=\min\limits_{x,y\in X}{\{\dist(x,y)\}}, \mbox{ где } x\ne y.$$
Обозначим через $e(X)$ число пар $(x,y)$ в $X$ у которых $\dist(x,y)=\psi(X)$. (Иными словами, $e(X)$ это число   ребер у графа $\cg(X)$)  Определим максимальное контактное число (число ребер) для заданного числа точек. 
$$
\ko({\bf M},N):=\max\limits_{X\in {\bf M}, |X|=N}{e(X)}. 
$$
В {\it задаче о наибольшем количестве контактов} спрашивается как найти числа $\ko({\bf M},N)$ или хотя бы оценить их?

В случае когда ${\bf M}={\Bbb S}^2$ будем обозначать величину $\ko({\bf M},N)$ через $\ko_N$. 

Это определение можно обобщить. Рассмотрим $N$ не пересекающихся кругов (сферических шапочек) диаметра $d$ на сфере ${\Bbb S}^2$. Обозначим через $\ko_N(d)$ максимально возможное число касаний этих кругов. 

Заметим, что 
$$
\ko_N=\max\limits_{d\leqslant d_N}{\ko_N(d)}.
$$

В работе  \cite{P12} рассматривалась эта задача для $N=12$ и $d=60^\circ$. Было доказано, что $\ko_N(d)=24.$

Начнем с простейших свойств множеств с максимальным числом контактов. 

\begin{lemma} Пусть $\Gamma_N$ обозначает контактный граф $\cg(X)$ множества $X$ на сфере, с $|X|=N$, у которого число ребер равно $\ko_N$. Тогда $\Gamma_N$ является планарным графом и степень любой вершины этого графа равна $2, 3, 4$ или $5$. 
\end{lemma}
\begin{proof} Планарность следует из Предложения 2.1. 

У графа $\Gamma_N$ не может быть ни изолированных ($\deg(v)=0$), ни висячих вершин 	($\deg(v)=1$), так как в противном случае мы можем придвинуть эту вершину к другим и увеличить число контактов до двух. 
 
Неравенство $\deg(v)<6$ следует из известного факта, что у равнобедренного сферического треугольника 
 $ABC$ с $|AB|=|AC|=d$ и $|BC|\ge d$ угол $\angle BAC > \pi/3=60^\circ$.  Предположим, что у  вершины $v$, $\deg(v)\geqslant 6$. Тогда   $60^\circ\deg(v)\geqslant 360^\circ$. С другой стороны, сумма углов при вершине $v$ равна $360^\circ$, а поскольку каждый угол больше чем $60^\circ$, то $360^\circ>60^\circ\deg(v)$,   - противоречие. Отсюда вытекает требуемое неравенство.
\end{proof}

Очевидно, что $\ko_2=1$.  Для больших $N$ имеет место следующее неравенство.  
\begin{thm} Пусть $N>2$. Тогда $$\ko_N\leqslant 3N-6$$ и равенство достигается только при $N=3,4,6,12$.
\end{thm}
\begin{proof} Рассмотрим контактный граф $\Gamma_N$ с максимальным числом контактов. Пусть у $\Gamma_N$  на сфере будет $F$ граней. Тогда по формуле Эйлера получаем 
$$
N - \ko_N +F=2.
$$  
Поскольку у каждой грани не менее трех сторон, то 
$F\leqslant2\ko_N/3$. Из этого неравенства и формулы Эйлера вытекает требуемое. 

Если $N=3,4,6,12$, то равенство соответственно достигается для правильного треугольника на экваторе, правильного тетраэдра, правильного октаэдра и правильного икосаэдра. 

Заметим, что равенство получается только в том случае, если все грани вложения $\Gamma_N$ на сферу являются конгруэнтными правильными сферическими треугольниками. Степени всех вершин должны быть равны. Но тогда при $N>3$ вершины задают правильные многогранники с треугольными гранями. Известно, что всего у трех из пяти правильных многогранников грани треугольные и они перечислены выше.   
\end{proof}

Поскольку из теоремы следует, что $\ko_5 < 9$, то получаем следующее утверждение: 
\begin{cor} $\ko_5=8$. Такое число получается когда одна точка находится в ``северном полюсе'', а оставшиеся четыре лежат на экваторе и образуют квадрат.
\end{cor}

Таким образом, мы знаем решение задачи о максимальном количестве контактов на сфере для $N\leqslant 6$ и $N=12$. Разберем теперь случаи  $7\leqslant N \leqslant 11$.

Множество с максимальным количеством контактов не обязано быть из $J_N$, т. е. его контактный граф не обязан быть неприводимым. Обозначим через $I_N$ множество состоящее из вершин правильного икосаэдра на сфере из которых удалено $12-N$ соседних вершин. Это означает, что  $I_{11}$ это множество вершин икосаэдра без одной вершины, $I_{10}$ - множество $I_{12}$ из которого удалены две вершины соединенные ребром, а   $I_{9}$  это вершины икосаэдра без трех вершин, лежащих на одной грани. Несложно видеть, что $e(I_N)=3N-9$ при  $N<11$. 

Поскольку $e(I_{10})=21$, то $\ko_{10}\geqslant 21$. С другой стороны, в таблице 7.5 - неприводимых графов для $N=10$, графами с наибольшим числом ребер являются 7.5.28 и 7.5.29 с  $e=20.$ (Заметим, что эти графы, полученные удалением двух не соседних вершин икосаэдра, не являются максимальными для задачи Таммеса.)  Следовательно, при $N=10$ максимальное число ребер достигается на графе, который является приводимым. 

Определим еще одну величину $\ko_N^*$ - максимальное число ребер у контактного неприводимого графа с $N$ вершинами, 
$$
\ko^*_N:=\max\limits_{X\in J_N}{e(X)}. 
$$ 

\begin{thm}
	\begin{enumerate}
	\item $\ko^*_7=\ko_7=12;$
	\item  $\ko^*_8=\ko_8=16;$
	\item  $\ko^*_9=\ko_9=18;$
	\item   $\ko^*_{10}=20, \, \ko_{10}=21;$
	\item   $ \ko^*_{11}=\ko_{11}=25.$ 
\end{enumerate}
\end{thm}

Доказательство теоремы следует из двух лемм о контактных графах. 

\begin{lemma} Пусть  $X$ - конечное множество на сфере ${\Bbb S}^2$. Если любая грань  контактного графа $\cg(X)$ 
является  треугольником или четырехугольником, то этот граф неприводим. 	 
\end{lemma}
\begin{proof} Заметим, что  граф $\cg(X)$ на сфере является приводимым если у него найдется невыпуклая грань. Однако, у контактного графа все грани являются равносторонними, и в нашем случае это либо равносторонний треугольник, либо ромб. В обоих случаях это выпуклые многоугольники.  
\end{proof}

\begin{lemma} Пусть  $X\subset{\Bbb S}^2$ и $|X|=N,\, N>6$. Предположим, что $e(X)\geqslant 3N-8$. Тогда контактный граф $\cg(X)$ является неприводимым. 
\end{lemma}

\begin{proof}  Из формулы Эйлера следует, что у графа с $N$ вершинами на сфере, все грани которого являются треугольниками, число ребер $e$ равно $3N-6$. По условию $e=3N-6, \,  3N-7$ или $3N-8$. В первом случае, у нас все грани треугольные, а во втором одна из граней четырехугольная. В обоих случаях Лемма 5.2 доказывает неприводимость контактного графа.
	
Остается третий случай, когда $e=3N-8$. Здесь возможно два подслучая: либо имеется ровно две четырехугольных грани, но тогда опять можно применить Лемму 5.2, либо одна пятиугольная грань $ABCDE$.

Предположим, что граф $\cg(X)$ приводим. Тогда сферический равносторонний пятиугольник  $ABCDE$ невыпуклый. Без ограничения общности, можно считать, что внутренний  угол вершины $A$ больше или равен  $180^\circ$, а все остальные углы меньше $180^\circ$. Тогда $A$ может быть соединено ребром только с одной вершиной $F$. Это следует из того факта, что все углы больше  $60^\circ$ и в случае если бы были еще вершины, кроме $F$, $B$ и $E$,  соединенные с $A$, то внешний угол $A$ был бы больше $180^\circ$. Поскольку все грани $\cg(X)$, кроме $ABCDE$, - треугольные, то все грани, которые сходятся в  вершине $F$ являются треугольниками и $F$  соединена ребрами с $B$ и $E$, см. рисунок.

Предположим, что степень вершины $F$ равна $m$. Тогда $m=3,4$ или $5$. Из этого следует, что углы треугольных граней соответственно равны $120^\circ$, $90^\circ$ или $72^\circ$. Первых двух случаев не может быть при $N>6$, а в третьем -- треугольные грани будут такими же как у правильного икосаэдра. 
Тогда, все внутренние углы пятиугольника $FBCDE$ равны $144^\circ$, а стало быть и вершина $A$ соединена с вершинами $C$ и $D$. Это противоречие завершает доказательство леммы.     
\end{proof}

\medskip 

\begin{center}
\begin{picture}(320,140)(-80,-70)
\put(-50,-65){ Рис. Случай, когда $ABCDE$ невыпуклый.}

\put(10,-20){\circle*{5}}

\put(90,-40){\circle*{5}}
\put(70,20){\circle*{5}}

\put(130,40){\circle*{5}}

\put(10,40){\circle*{5}}

\put(70,60){\circle*{5}}

\thicklines
\put(10,-20){\line(4,-1){80}}
\put(10,-20){\line(0,1){60}}

\put(70,20){\line(3,1){60}}

\put(70,20){\line(-3,1){60}}

\put(90,-40){\line(1,2){40}}

\put(70,60){\line(0,-1){40}}
\put(70,60){\line(3,-1){60}}
\put(70,60){\line(-3,-1){60}}

\put(-5,-21){$D$}
\put(97,-41){$C$}
\put(58,28){$A$}
\put(134,44){$B$}
\put(73,65){$F$}
\put(-3,45){$E$}
 
\end{picture}
\end{center}

\begin{proof} Докажем теперь теорему. При $N=8$ и $N=11$ у графов 7.3.4 и 7.6.22 число ребер равно  $3N-8$. Тогда Лемма 5.3 гарантирует, что контактный граф с максимальным числом ребер является неприводимым и поэтому $\ko_N=\ko^*_N$.  В этих случаях, множества с максимальным количеством контактов единственны с точностью до изометрии. 
	
Для оставшихся $N$ множества $I_N$ дают примеры контактных графов с $3N-9$  ребрами. (Более того, для $N=7$ и $N=9$ у максимальных графов 7.2.2 и 7.4.7 тоже $3N-9$ ребер.) Докажем, что это правильный ответ.  От противного, если предположить что у контактных графов бывает больше ребер, т. е. $e(X)\geqslant 3N-8$, то по Лемме 5.3 графы  $\cg(X)$ являются неприводимыми. Однако, в таблицах 7.2, 7.4 и 7.5 графов с таким числом ребер нет.  
\end{proof}

\section{Задачи о неприводимых контактных графах}

В работе Л. Данцера \cite{Dan} имеется множество открытых вопросов, связанных с неприводимыми и Д-неприводимыми контактными графами. Наши работы \cite{MT,MT2013,MT14} тоже поднимают  большое число вопросов.  Остановимся на некоторых из них. 

Зафиксируем абстрактный граф $G$. Предположим, что найдется такое множество $X\subset{\Bbb S}^2$, что его контактный граф  является неприводимым и изоморфен $G$.  
Возможны два варианта:\\
 (i)  множество $X$ единственное с точностью до изометрии;\\
 (ii) имеется степень свободы для $X$, т. е. на сфере существует $k$-параметрическое семейство множеств $X$ с $k\geqslant 1$. 

\medskip  

Доказано, что  для $6\leqslant N\leqslant14$ и $N=24$ к (i) типу относятся максимальные графы (т. е. контактные графы, задающие решение проблемы Таммеса). Возникает вопрос:

\medskip

\noindent{\bf 1.} {\it Верно ли, что все максимальные графы лежат в классе  $(i)$?}

\medskip

Этот вопрос можно усилить. Мы уже отмечали, что в торической проблеме Таммеса  при $N=7$ оказалось, что имеется три неизоморфных максимальных графа \cite{MN}.  

\medskip

\noindent
{\bf 2.} {\it Верно ли, что на сфере для заданного $N>5$ имеется единственный (с точностью до изоморфизма) максимальный граф?}

\medskip

Для $6\leqslant N\leqslant10$,  $N=13$ и $N=14$ максимальные графы получаются  из графов типа (ii) добавлением ребер. Вместе с тем, для $N=11,12$ максимальные графы изолированы и не получаются таким образом. В связи с этим появляется такой вопрос:

\medskip

\noindent
{\bf 3.} {\it Являются ли случаи $N=11,12$ исключением и все остальные максимальные графы получаются из класса   $(ii)$?}

\medskip

Заметим, что в классе $(i)$ лежат не только максимальные графы. Например, к нему относятся графы 7.5.11, 7.5.21, 7.5.28, 7.5.29, 7.6.1. 7.6.5, 7.6.6, 7.6.15 и еще несколько графов из таблицы 7.6. Однако, большинство графов обладают степенью свободы. Л. Данцер \cite[p. 21]{Dan} обсуждая графы типа (ii), отмечает, что во всех рассмотренных им примерах, меняя параметры так, чтобы увеличивалось $d=\psi(X)$ мы достигнем максимума, когда к графу добавятся новые ребра. Однако, далее он пишет, что не решается сформулировать это наблюдение в качестве гипотезы. А мы все-таки зададим вопрос: 

\medskip

\noindent
{\bf 4.} {\it Пусть граф $G$ относится к классу $(ii)$. Предположим, что $X\subset{\Bbb S}^2$ такое, что $\cg(X)=G$. Верно ли, что можно  немного так изменить $X$ на $X'$, что $\cg(X')=G$ и $\psi(X')>\psi(X)$?}

\medskip

По крайней мере, на один из вопросов Данцера \cite[Question 5, p. 65]{Dan} мы можем ответить. Он спрашивает:  {\it``Имеется ли на сфере такое множество $X$, что $|X|<12$, $\cg(X)$ является неприводимым контактным графом и $\psi(X)<d_{12}=\arccos{1/\sqrt{5}}$?''}

\medskip

Дело в том, что среди рассмотренных Данцером Д-неприводимых графов таких не оказалось. Однако, в нашем списке графы с $\psi(X)<d_{12}$ появляются уже при $N=9$.  Это граф 7.4.3. Таких графов много для $N=10$ и такие все, кроме максимального, для $N=11$. 

 Поскольку максимальный граф является неприводимым, то 
$$
d_N=\max\limits_{X\in J_N}{\psi(X)}. 
$$

Рассмотрим теперь и минимум. 
$$
\delta_N=\min\limits_{X\in J_N}{\psi(X)}. 
$$
Тогда вопрос Данцера можно переформулировать как:   ``Найти минимальное $N$ при котором $\delta_N<d_{12}$.''   (Ответ на этот вопрос, как мы отмечали выше, $N=9$.) По аналогии с проблемой Таммеса здесь возникает такая задача  

\medskip

\noindent
{\bf 5.} {\it Найти все $N$--точечные  конфигурации  $X$ на ${\mathbb S}^2$, чтобы  $\psi(X)=\delta_N$.}

\medskip

Таблицы 7.1--7.6 позволяют ответить на этот вопрос для $N<12$, но поскольку $d_{min}$ там найдены численно, необходим более детальный геометрический анализ минимальных конфигураций. 

\medskip

В предыдущем разделе мы рассматривали $K_N^*$ -- максимальное число ребер у неприводимого контактного графы, а здесь мы рассмотрим минимальное число ребер.  

\medskip

\noindent
{\bf 6.} {\it Найти нижние оценки на величину $\kappa_N$, где} 
$$
\kappa_N:=\min\limits_{X\in J_N}{e(X)}. 
$$

Из таблиц 7.1--7.6 вытекает следующая теорема:
\begin{thm}
	\begin{enumerate}
	\item $\kappa_6=9$;
	\item $\kappa_7=11;$
	\item  $\kappa_8=12;$
	\item  $\kappa_9=12;$
	\item   $\kappa_{10}=14;$
	\item   $\kappa_{11}=15.$ 
\end{enumerate}
\end{thm}

\medskip

В заключение, остановимся еще на одной задаче из списка Данцера \cite[p. 64]{Dan}. Пусть точки $X=\{x_1,\ldots,x_N\}$ заданы сферическими координатами $(\theta_i, \varphi_i)$.  Без ограничения общности можно считать, что $\theta_1=\pi/2, \, \varphi_1=0$ и $\theta_2=\pi/2$. Тогда сферические координаты задают пространство конфигураций $\Pi_N$ размерности $2N-3$.  

Поскольку, для каждого $X$ на сфере у нас определена величина $\psi(X)$, то задана функция   $\psi:\Pi_N\to {\Bbb R}$. Возникает вопрос: 

\medskip

\noindent
{\bf 7.} {\it Найти условия при которых $X$  является  максимумом функции $\psi$ на $\Pi_N$.  }

\medskip

Нам кажется, что  эта задача довольно сложная.

\section{Приложение: Список всех неприводимых контактных графов для $N\leqslant 11$}
Приведем здесь основной результат нашей работы \cite{MT2013}.  
\begin{thm} Список всех неприводимых контактных графов для $N=6,7,8,9,10,11$
на сфере ${\Bbb S}^2$ приводится в таблицах 7.1--7.6.
\end{thm}

 В таблицах $*$
означает, что соответствующий неприводимый граф является также
и Д-неприводимым, а  $**$ означает, что этот граф - максимальный. 
Заметим, что Л. Данцер перечислил Д-неприводимые графы до $N=10$, и поэтому в таблице для $N=11$ у нас не отмечены Д-неприводимые графы.

В таблицах также показаны и предельные значения $d$, $d_{min}\leqslant
d \leqslant d_{max}$. (Однако, отметим, что величины $d_{min}$
и $d_{max}$ найдены численно и могут немного отличаться от истинных.)

\subsection{Неприводимые графы для 6 вершин.}
\begin{tabular}{ccc}
$N$ & $d_{min}$ & $d_{max}$ \\
$1*$ & $1.4274$ & $1.5708$ \\
$2**$ & $1.5708$ & $1.5708$ \\

\end{tabular}

\begin{tabular}{cccc}
  &
\includegraphics[clip,scale=0.4]{pics/irr-6.mps}  
& \; \; \; &
\includegraphics[clip,scale=0.2]{pics/irr-4.mps}  
\end{tabular}

\subsection{Неприводимые графы для 7 вершин.}
\begin{tabular}{ccc}
$N$ & $d_{min}$ & $d_{max}$ \\
$1*$ & $1.34978$ & $1.35908$ \\
$2**$ & $1.35908$ & $1.35908$ \\
\end{tabular}

\begin{tabular}{cc}
\includegraphics[clip,scale=0.5]{pics/seven1.mps} 
~~~~
\includegraphics[clip,scale=0.5]{pics/seven2.mps}
\end{tabular}

\subsection{Неприводимые графы для 8 вершин.}

\begin{tabular}{ccc}
$N$ & $d_{min}$ & $d_{max}$ \\
$1$ & $1.17711$ & $1.18349$ \\
$2*$ & $1.28619$ & $1.30653$ \\
$3*$ & $1.23096$ & $1.30653$ \\
$4**$ & $1.30653$ & $1.30653$ \\

\end{tabular}

\begin{tabular}{cc}
\includegraphics[clip,scale=0.5]{pics/eight1.mps} 
~~
\includegraphics[clip,scale=0.5]{pics/eight2.mps}
~~
\includegraphics[clip,scale=0.5]{pics/eight3.mps}
~~
\includegraphics[clip,scale=0.5]{pics/eight4.mps}
\end{tabular}

\subsection{Неприводимые графы для 9 вершин.}
\begin{tabular}{ccc}
$N$ & $d_{min}$ & $d_{max}$ \\
$1$ & $1.14099$ & $1.14143$ \\
$2*$ & $1.22308$ & $1.23096$ \\
$3$ & $1.10525$ & $1.14349$ \\
$4$ & $1.17906$ & $1.18106$ \\
$5$ & $1.15448$ & $1.17906$ \\
$6$ & $1.17906$ & $1.17906$ \\
$7**$ & $1.23096$ & $1.23096$ \\
$8$ & $1.15032$ & $1.18106$ \\
$9*$ & $1.10715$ & $1.14342$ \\
$10$ & $1.17906$ & $1.18428$ \\
\end{tabular}

\begin{tabular}{cc}
\includegraphics[clip,scale=0.5]{pics/nine1.mps} 
~~
\includegraphics[clip,scale=0.5]{pics/nine2.mps} 
~~
\includegraphics[clip,scale=0.5]{pics/nine3.mps} 
~~
\includegraphics[clip,scale=0.5]{pics/nine4.mps} 
~~
\includegraphics[clip,scale=0.5]{pics/nine5.mps} \\
\includegraphics[clip,scale=0.5]{pics/nine6.mps} 
~~
\includegraphics[clip,scale=0.5]{pics/nine7.mps} 
~~
\includegraphics[clip,scale=0.5]{pics/nine8.mps} 
~~
\includegraphics[clip,scale=0.5]{pics/nine9.mps} 
~~
\includegraphics[clip,scale=0.5]{pics/iv_nine1.mps} 
~~
\end{tabular}

\subsection{Неприводимые графы для 10 вершин.}

\begin{tabular}{ccccccc}
$N$ & $d_{min}$ & $d_{max}$ & & $N$ & $d_{min}$ & $d_{max}$ \\
$1$ & $1.0839$ & $1.09751$ & & $2$ & $1.08161$ & $1.08439$ \\
$3$ & $1.03067$ & $1.04695$ & & $4$ & $1.10715$ & $1.0988$  \\
$5$ & $1.07529$ & $1.09431$ & & $6$ & $1.09386$ & $1.12285$ \\
$7*$ & $1.15278$ & $1.15448$ & & $8$ & $1.10012$ & $1.10801$ \\
$9$ & $1.06344$ & $1.07834$ & & $10*$ & $1.15074$ & $1.15191$ \\
$11$ & $1.0843$ & $1.08442$ & & $12$ & $1.10055$ & $1.10889$ \\
$13$ & $1.09504$ & $1.10429$ & & $14$ & $1.06032$ & $1.09604$ \\
$15$ & $1.06278$ & $1.1098$ & & $16$ & $1.09567$ & $1.10715$ \\
$17**$ & $1.15448$ & $1.15448$ & & $18$ & $0.99865$ & $1.0467$ \\
$19$ & $1.0843$ & $1.0844$ & & $20$ & $1.08334$ & $1.09547$ \\
$21*$ & $1.15341$ & $1.15341$ & & $22$ & $1.0988$ & $1.10608$ \\
$23*$ & $1.14372$ & $1.15191$ & & $24$ & $1.09249$ & $1.1098$ \\
$25*$ & $1.15191$ & $1.15245$ & & $26$ & $1.09658$ & $1.10977$ \\
$27*$ & $1.15191$ & $1.15191$ & & $28*$ & $1.10715$ & $1.10715$ \\
$29*$ & $1.10715$ & $1.10715$ & & $30$ & $1.15103$ & $1.15341$  \\ 
\end{tabular}
\newpage
\begin{tabular}{c}
\includegraphics[clip,scale=0.5]{pics/ten1.mps} 
~~
\includegraphics[clip,scale=0.5]{pics/ten2.mps} 
~~
\includegraphics[clip,scale=0.5]{pics/ten3.mps} 
~~
\includegraphics[clip,scale=0.5]{pics/ten4.mps} 
~~
\includegraphics[clip,scale=0.5]{pics/ten5.mps} 

\includegraphics[clip,scale=0.5]{pics/ten6.mps} 
\end{tabular}
\begin{tabular}{c}
\includegraphics[clip,scale=0.5]{pics/ten7.mps} 
~~
\includegraphics[clip,scale=0.5]{pics/ten8.mps} 
~~
\includegraphics[clip,scale=0.5]{pics/ten9.mps} 
~~
\includegraphics[clip,scale=0.5]{pics/ten10.mps} 
~~
\includegraphics[clip,scale=0.5]{pics/ten11.mps} 
~~
\includegraphics[clip,scale=0.5]{pics/ten12.mps} \\
\end{tabular}
\begin{tabular}{c}
\includegraphics[clip,scale=0.5]{pics/ten13.mps} 
~~
\includegraphics[clip,scale=0.5]{pics/ten14.mps} 
~~
\includegraphics[clip,scale=0.5]{pics/ten15.mps} 
~~
\includegraphics[clip,scale=0.5]{pics/ten16.mps} 
~~
\includegraphics[clip,scale=0.5]{pics/ten17.mps} 
~~
\includegraphics[clip,scale=0.5]{pics/ten18.mps} \\
\end{tabular}
\begin{tabular}{c}
\includegraphics[clip,scale=0.5]{pics/ten19.mps} 
~~
\includegraphics[clip,scale=0.5]{pics/ten20.mps} 
~~
\includegraphics[clip,scale=0.5]{pics/ten21.mps}  
~~
\includegraphics[clip,scale=0.5]{pics/ten22.mps} 
~~
\includegraphics[clip,scale=0.5]{pics/ten23.mps} 
~~
\includegraphics[clip,scale=0.5]{pics/ten24.mps} \\
\end{tabular}
\begin{tabular}{c}
\includegraphics[clip,scale=0.5]{pics/ten25.mps} 
~~
\includegraphics[clip,scale=0.5]{pics/ten26.mps} 
~~
\includegraphics[clip,scale=0.5]{pics/ten27.mps} 
~~
\includegraphics[clip,scale=0.5]{pics/ten28.mps} 
~~
\includegraphics[clip,scale=0.5]{pics/ten29.mps} 
~~
\includegraphics[clip,scale=0.5]{pics/iv_ten1.mps} 
\end{tabular}

\medskip

\subsection{Неприводимые графы для 11 вершин.}
\begin{tabular}{ccccccc}
$N$ & $d_{min}$ & $d_{max}$ & & $N$ & $d_{min}$ & $d_{max}$ \\
$1$ & $1.05601$ & $1.05602$ & & $2$ & $1.0538$ & $1.05842$ \\
$3$ & $1.05834$ & $1.05842$ & & $4$ & $1.04765$ & $1.05455$  \\
$5$ & $1.06975$ & $1.06974$ & & $6$ & $1.06306$ & $1.06308$ \\
$7$ & $1.0522$ & $1.06131$ & & $8$ & $1.06621$ & $1.06846$ \\
$9$ & $1.0538$ & $1.05531$ & & $10$ & $1.0795$ & $1.07961$ \\
$11$ & $1.05331$ & $1.0737$ & & $12$ & $1.07163$ & $1.07197$ \\
$13$ & $1.0404$ & $1.06635$ & & $14$ & $1.04759$ & $1.05637$ \\
$15$ & $1.06974$ & $1.06974$ & & $16$ & $1.02726$ & $1.06117$ \\
$17$ & $1.04712$ & $1.06167$ & & $18$ & $1.06043$ & $1.06209$ \\
\end{tabular}
\begin{tabular}{ccccccc}
$N$ & $d_{min}$ & $d_{max}$ & & $N$ & $d_{min}$ & $d_{max}$ \\
$19$ & $1.05386$ & $1.05947$ & & $20$ & $1.05846$ & $1.05882$ \\
$21$ & $1.0632$ & $1.0636$ & & $22**$ & $1.10715$ & $1.10715$ \\
$23$ & $1.05388$ & $1.06537$ & & $24$ & $1.05375$ & $1.0737$ \\
$25$ & $1.06167$ & $1.0636$ & & $26$ & $1.06506$ & $1.06673$ \\
$27$ & $1.04636$ & $1.05882$ & &  $28$ & $1.05426$ & $1.06822$ \\
$29$ & $1.07832$ & $1.07836$ & & $30$ & $1.07886$ & $1.07962$ \\
$31$ & $1.05429$ & $1.06105$ & & $32$ & $1.00523$ & $1.05671$ \\
$33$ & $1.061$ & $1.06117$ & & $34$ & $1.02751$ & $1.05828$ \\
$35$ & $1.05447$ & $1.06679$ & & $36$ & $1.0561$ & $1.05627$ \\
$37$ & $1.05431$ & $1.05827$ & & $38 (iv)$ & $1.0064$ & $1.03613$ \\
\end{tabular}

\medskip

\begin{tabular}{c}
\includegraphics[clip,scale=0.5]{pics/eleven1.mps} 
~~
\includegraphics[clip,scale=0.5]{pics/eleven2.mps} 
~~
\includegraphics[clip,scale=0.5]{pics/eleven3.mps} 
~~
\includegraphics[clip,scale=0.5]{pics/eleven4.mps} 
~~
\includegraphics[clip,scale=0.5]{pics/eleven5.mps} 
~~
\includegraphics[clip,scale=0.5]{pics/eleven6.mps} \\
~~
\includegraphics[clip,scale=0.5]{pics/eleven7.mps} 
~~
\includegraphics[clip,scale=0.5]{pics/eleven8.mps} 
~~
\includegraphics[clip,scale=0.5]{pics/eleven9.mps} 
~~
\includegraphics[clip,scale=0.5]{pics/eleven10.mps} 
~~
\includegraphics[clip,scale=0.5]{pics/eleven11.mps} 
~~
\includegraphics[clip,scale=0.5]{pics/eleven12.mps} \\
~~
\includegraphics[clip,scale=0.5]{pics/eleven13.mps} 
~~
\includegraphics[clip,scale=0.5]{pics/eleven14.mps} 
~~
\includegraphics[clip,scale=0.5]{pics/eleven15.mps} 
~~
\includegraphics[clip,scale=0.5]{pics/eleven16.mps} 
~~
\includegraphics[clip,scale=0.5]{pics/eleven17.mps} 
~~
\includegraphics[clip,scale=0.5]{pics/eleven18.mps} \\
\end{tabular}
\begin{tabular}{c}
\includegraphics[clip,scale=0.5]{pics/eleven19.mps} 
~~
\includegraphics[clip,scale=0.5]{pics/eleven20.mps} 
~~
\includegraphics[clip,scale=0.5]{pics/eleven21.mps} 
~~
\includegraphics[clip,scale=0.5]{pics/eleven22.mps} 
~~
\includegraphics[clip,scale=0.5]{pics/eleven23.mps} 
~~
\includegraphics[clip,scale=0.5]{pics/eleven24.mps} \\
~~
\includegraphics[clip,scale=0.5]{pics/eleven25.mps} 
~~
\includegraphics[clip,scale=0.5]{pics/eleven26.mps} 
~~
\includegraphics[clip,scale=0.5]{pics/eleven27.mps} 
~~
\includegraphics[clip,scale=0.5]{pics/eleven28.mps} 
~~
\includegraphics[clip,scale=0.5]{pics/eleven29.mps} 
~~
\includegraphics[clip,scale=0.5]{pics/eleven30.mps} \\
~~
\includegraphics[clip,scale=0.5]{pics/eleven31.mps} 
~~
\includegraphics[clip,scale=0.5]{pics/eleven32.mps} 
~~
\includegraphics[clip,scale=0.5]{pics/eleven33.mps} 
~~
\includegraphics[clip,scale=0.5]{pics/eleven34.mps} 
~~
\includegraphics[clip,scale=0.5]{pics/eleven35.mps} 
~~
\includegraphics[clip,scale=0.5]{pics/eleven36.mps} \\
~~
\includegraphics[clip,scale=0.5]{pics/eleven37.mps} 
~~
\includegraphics[clip,scale=0.5]{pics/iv_eleven1.mps} 
~~
\end{tabular}

\medskip

\medskip

\medskip

\medskip

\medskip

\medskip

О. Р. Мусин, ИППИ РАН и UTB  (University of Texas at
Brownsville).

 {\it E-mail:} oleg.musin@utb.edu

\medskip

А. С. Тарасов, ИППИ РАН

{\it E-mail:} tarasov.alexey@gmail.com

\end{document}